\documentclass[12pt]{amsart}
\usepackage{amsaddr}
\usepackage{amsfonts,amsmath,amsxtra}
\usepackage{latexsym}
\usepackage{amssymb}
\usepackage{graphicx}
\usepackage{mathdots} 
\usepackage{tikz-cd}
\usepackage{amsmath,amssymb,amsthm,amsopn,amsfonts,hyperref,xcolor}
\usepackage[matrix,arrow,curve]{xy}

\newcommand{\RNum}[1]{\uppercase\expandafter{\romannumeral #1\relax}}

\newtheorem{te}{Theorem}[section]

\newtheorem{de}{Definition}[section]

\newtheorem{lem}{Lemma}[section]
\newtheorem{prop}{Proposition}[section]
\newtheorem{rem}{Remark}[section]

\newcommand\bqa{\begin{eqnarray}}
\newcommand\eqa{\end{eqnarray}}

\def\<{\langle}
\def\>{\rangle}

\newcommand{\CH}{{\mathcal H}}

\newcommand{\CJ}{{\mathcal J}}

\newcommand{\CV}{{\mathcal V}}
\newcommand{\CW}{{\mathcal W}}

\newcommand{\IC}{{\mathbb C}}

\newcommand{\IF}{{\mathbb F}}

\newcommand{\IQ}{{\mathbb Q}}

\newcommand{\IR}{{\mathbb R}}

\newcommand{\IZ}{{\mathbb Z}}

\newcommand{\fX}{{\mathfrak X}}

\numberwithin{equation}{section}
\begin{document}

\title{On explicit computation of Curtis homomorphisms for $GL(n, \IF_q)$}

\author{Xuantong Qu}
\address{School of Mathematical Sciences, University of Nottingham,\\ University Park, Nottingham NG7 2RD, UK}
\email{pmyxq@nottingham.ac.uk}

\begin{abstract}
    In this note we give explicit computations of certain types of Curtis homomorphisms and interpret them in terms of Gelfand-Tsetlin diagrams. Namely, this interpretation follows from  Gelfand-Tsetlin formulas for the $GL(n, \IF_q)$-Whittaker functions associated with principal series representations via exploiting the method in \cite{glo} for $GL(n,\IR)$-Whittaker functions. The explicit computations of Whittaker functions are based on computing the Gauss-Bruhat decomposition for certain elements in $GL(n, \IF_q)$.
\end{abstract}

\maketitle

\section{Introduction}
In this note we are concerned with the Gelfand-Graev representation of $GL(n, \IF_q)$, $n>1$. There are two equivalent ways of describing irreducible representations entering the Gelfand-Graev representations. In 1970, S.I. Gelfand \cite{gel} introduced Bessel functions associated to each irreducible representation entering Gelfand-Graev representation of group $GL(n, \IF_q)$ and gave explicit computations for certain cases of $GL(3, \IF_q)$. Another way exploits properties of the Hecke algebra $\CH$ (or $GL(n, \IF_q)-$endomorphism algebra) of the Gelfand-Graev representation. In 1993, Curtis proved the following factorization theorem (Theorem 4.2 in \cite{c1}) for the homomorphisms between $\overline{\IQ_{l}}$-algebras: 
\begin{equation}
    \begin{tikzcd}
\CH \arrow{r}{f_{T,\theta}} \arrow{dr}{f_T} & \overline{\IQ_{l}}\\
&\overline{\IQ_{l}}[T] \arrow{u}{\widehat{\theta}}
\end{tikzcd}
\end{equation}
where $T$ is a maximal torus in $GL(n, \IF_q)$, $\theta:T\rightarrow \overline{\IQ_{l}}$ is a character of the group $T$ and $\widehat{\theta}$ is the linear extension of $\theta$ to the group algebra: $\overline{\IQ_{l}}[T]$. 

\ 

Since $\overline{\IQ_{l}}$ is isomorphic to $\IC$ as a field, in the following sections we shall see that our computation can be implemented for both $\overline{\IQ_{l}}$ and $\IC$ interchangeably. Therefore throughout this note we shall regard the Hecke algebra $\CH$ and the group algebra of $T$ as algebras over $\IC$.

\ 

We refer the algebra homomorphism $f_T$ in (0.1) as $\textbf{Curtis homomorphism}$. In the factorization theorem, the Curtis homomorphism was constructed using the Green function over finite fields whose general theory was developed by Deligne and Lusztig \cite{dl} via an geometric approach of constructing all irreducible representations of $G$ using $l$-adic cohomology. It is a natural problem of computing $f_T$ explicitly which is directly related to explicit computations of Bessel functions over finite fields (defined for $GL(n, \IF_q)$ by Gelfand in \cite{gel}). Furthermore, it is expected by Curtis \cite{c2} that it is probable to find a combinatorial proof of the existence of $f_T$, which is also closely related to the explicit computations of structure constants of the Hecke algebra $\CH$. Following his own construction, Curtis computed $f_T$ explicitly when $G=SL_2$ in \cite{c1} and later with Shinoda \cite{cs} they gave explicit computation of $f_T$ on maximal parabolics for $G=GL_n$. We note that both explicit computations in \cite{c1} and \cite{cs} exploited properties of Green functions over finite fields and they are hard to generalize to other cases.

\ 

In this note we propose a new approach to the problem following the ideas of \cite{glo} for Archimedean Whittaker functions.  Namely we derive the Gelfand-Tsetlin formulas for $GL(n, \IF_q)$-Whittaker functions in Section 3 and 4 which share the same formalism and combinatorial interpretations in terms of Gelfand-Tsetlin diagrams with the $GL(n, \IR)$-Whittaker functions. As a consequence, we are able to give explicit computation of values of Curtis homomomorphisms for $G_n=GL(n, \IF_q)$ and $T$ being its maximal split tori (we shall fix $T=T_n$ to be the subgroup of diagonal matrices) for the case of maximal and minimal parabolics (defined in Section 1). The main results of this note are stated in Theorems 2.1-2.3 of Section 2, which illustrate that the Curtis homomophisms are indeed "sums" over the corresponding Gelfand-Tsetlin diagrams.

\

We first note that by work of Gelfand \cite{gel}, the homomorphisms $f_{T,\theta}$ are in one-to-one correspondence with the Bessel functions of the generic representations of $G_n$. We also note that for every character $\underline{\chi}_{n}$ of $T_n$, ($T_n$, $\underline{\chi}_{n}$) are in one-to-one correspondence with the principal series representations of $G_n$. Therefore the problem reduces to the explicit computation of Bessel functions corresponding to principal series representations.

\ 

In Section 1 we identify the Bessel functions (up to multiple of constants) with matrix coefficients of the corresponding principal series representations. We call these matrix coefficients $GL(n, \IF_q)$-Whittaker functions due to similarity of constructions for $GL_n$ over Archimedean local
fields \cite{glo}. Namely, a recursive formula for $GL(n, \IR)$-Whittaker functions with respect to $n$ was derived and a combinatorial interpretation of this recursion in terms of Gelfand-Tsetlin diagrams was explained. We propose (in Section 3 and 4) that the same result (in terms of formalism and combinatorial interpretation) in \cite{glo} also holds for their finite field counterparts defined in Section 1. Thus, we refer to as Gelfand-Tsetlin formulas in this note. Explicit formulas of Curtis homomophisms would easily follow once the Whittaker (Bessel) functions were explicitly computed by the Gelfand-Tsetlin formulas. 

\ 

We also remark that these two types of values of Whittaker functions for the maximal and minimal parabolics were considered as Kloosterman sums previously. In particular, the values of Whittaker functions on anti-diagonals were called long element Kloosterman sums in \cite{gb}, or big cell Kloosterman sums in \cite{bfg,fr}.  $GL_3$-long element Kloosterman sums was computed in \cite{bfg}. An algorithm of computing $GL_n$-long element Kloosterman sums was described in \cite{fr}, however, in general it is hard to have explicit formulas for arbitrary $n$ following this algorithm. Combinatorial interpretation of Kloosterman sums in terms of Gelfand-Tsetlin diagrams as well as recursion over the ranks were not considered previously as far as the author's knowledge.

\ 

The plan of this note is as follows. We first review the notions of Bessel functions following the work of Gelfand \cite{gel} and identify them with Whittaker functions defined over finite fields. The main results in this note are explicit formulas of Curtis homomorphisms (Theorems 2.1-2.3) and the Gelfand-Tsetlin formulas of $GL(n, \IF_q)$-Whittaker functions(Theorem 3.1, 3.2 and 4.1). The values of Curtis homomorphisms are products of recursion kernels $\widetilde{Q}$, therefore they are "sums" over the corresponding Gelfand-Tsetlin diagrams. In Section 3 and 4 we give Gelfand-Tsetlin formulas for the values of Whittaker functions on maximal and minimal parabolics respectively. The formulas are expressed recursively over the ranks,  namely in both cases a $GL_n$-Whittaker function could be expressed in terms of $GL_{n-1}$-Whittaker functions and the recursion kernel is summed over a Gelfand-Tsetlin diagram. Therefore by resolving the recursive relations obtained we have a explicit formula of values of Whittaker functions and its combinatorial interpretation, as explained in Theorem 3.1, 3.2 and 4.1. In the last section we prove the proposition serving as key step while we derive recursions of anti-diagonal matrices, which has the same forms as its Archimedean counterparts (see Part 2 of \cite{glo}). We show that this is not merely a coincidence, and verify the appropriate expressions using Cauchy-Binet formula, which works for arbitrary fields.

\renewcommand{\abstractname}{Acknowledgements}
\begin{abstract}
  The author is grateful to his PhD supervisor S. Oblezin for fruitful discussions and constructive comments during the preparation of this paper and all the supports for the author's PhD study. The PhD project of the author is supported by University of Nottingham/UKRI International Scholarship. The UKRI scholarship is part of the EPSRC DTP grant EP/T517902/1-2603416.
\end{abstract}

\section{Preliminaries: Whittaker models, Bessel functions and Whittaker functions}
Let $\IF_{q}$ be a finite field with $q$ elements, where $q$ is a power of an odd prime, and $\psi:\IF_{q}\rightarrow \IC^{\times}$ be a nontrivial additive character. Let $U_{n}$ be the upper-triangular unipotent subgroup of general linear group $G_n=GL(n,\IF_{q})$, therefore $\psi$ defines a character on $U_{n}$ by
$$
\psi\begin{pmatrix}
1 & a_1 &* &\cdots &* \\
  &  1 &a_{2} &\ddots &\vdots \\
&  & \ddots & \ddots &*\\
  &   & &1 & a_{n-1}\\
  & & & & 1
\end{pmatrix} = \psi\left(\sum_{i=1}^{n-1}a_i\right).
$$

Consider the Gelfand-Graev representation 
$$
Ind_{U_n}^{G_n}\psi=\{f:G_n \rightarrow \IC \mid f(ug)=\psi(u)f(g) \quad \forall u\in U_n, \forall g\in G_n\},
$$
where $Ind_{U_n}^{G_n}\psi$ is a subrepresentation of the right regular representation of $G_n$. Steinberg \cite{St} showed that $Ind_{U_n}^{G_n}\psi$ is multiplicity-free, i.e., every irreducible subrepresentations of $Ind_{U_n}^{G_n}\psi$ appear exactly once. Following \cite{gel}, we call an irreducible representation $\pi$ of $G_n$ generic if it is isomorphic to a subrepresentation of $Ind_{U_n}^{G_n}\psi$, and we denote the unique subspace of $Ind_{U_n}^{G_n}\psi$ which is isomorphic to $\pi$ by $\CW(\pi,\psi)$. We call $\CW(\pi,\psi)$ the $\textbf{($\psi$-)Whittaker model}$ of $\pi$, with respect to character $\psi$.

There are two ways of descibing irreducible representations entering Gelfand-Graev represenation. Gelfand \cite{gel} showed that for any irreducible generic representation $\pi$ of $G_n$, there exists a unique element of $\CW(\pi,\psi)$ satisfying the following properties, which is called the $\textbf{Bessel function}$ of $\pi$. 

\begin{de}
For any irreducible generic representation $\pi$ of $G_n$, there exists an unique element $J_{\pi}\in \CW(\pi,\psi)$ such that:
\begin{itemize}
    \item [(1)] $J_{\pi}(I_n)=1$, where $I_n$ is the identity matrix;
    \item[(2)]$J_{\pi}(u_1gu_2)=\psi(u_1u_2)J_{\pi}(g)$ for every $u_1, u_2 \in U_n$, $g\in G_n$. (This property is also called $U_n$-bi-equivariance.)
\end{itemize}
\end{de}

\ 

Let $T_n$ be the subgroup of diagonal matrices and $W_n$ is the subgroup of permutation matrices (which is isomorphic to the Weyl group of $G_n$), we have the Bruhat decomposition of $G_n$ following \cite{gel}
\begin{align}
    G_n=U_nT_nW_nU_n.
\end{align}
Let $B_{n}^{-}$ be a Borel subgroup of $G_n=GL(n,\IF_{q})$ consisting of all lower-triangular matrices, therefore (1.1) is equivalent to the following:
\begin{align}
    G_n=\bigsqcup\limits_{w\in W_n}B_{n}^{-}wU_n.
\end{align}

From (1.1) the ($U_n,U_n$)-double cosets in $G_n$ are parametrized by $N_n=T_nW_n$ (which is the subgroup consisting of all monomial matrices). Therefore by the properties of Bessel functions stated above, we have the following theorem for the support of Bessel function $J_{\pi}$:
\begin{te}(\cite{gel}, Proposition 4.9)
$J_{\pi}$ is supported on double cosets of the form
$$
U_n g_{n_1,\cdots, n_s}(c_1,\cdots,c_s) U_n,
$$
where for $n_1,\cdots,n_s>0$ with $n_1+\cdots+n_s=n$ and $c_1,\cdots,c_s\in \IF_{q}^{\times}$,
$$
g_{n_1,\cdots, n_s}(c_1,\cdots,c_s)=\begin{pmatrix}
   0&\cdots &0 &c_1I_{n_1} \\
    \vdots &\iddots &c_2I_{n_2} &0 \\
   0& \iddots &\iddots  &\vdots\\
  c_sI_{n_s}   &0 &\cdots & 0

\end{pmatrix},
$$

where $I_n$ denotes the identity $n\times n$ matrix.
\end{te}

Therefore by the property (2) of the Bessel function $J_{\pi}$ mentioned above we may concentrate on the values of $J_{\pi}$ on $g_{n_1,\cdots, n_s}(c_1,\cdots,c_s)$.

\ 

In this note we concentrate on two cases of the support of the Bessel functions, namely matrices of the form $g_{1,n-1}(c_1,1)$ and $g_{n-1,1}(c_1,1)$; and matrices of the form $g_{1,\cdots,1}(c_1,\cdots,c_n)$, i.e., anti-diagonal matrices. Since each partition of $n=n_1+\cdots+n_s$ determines a parabolic subgroup of $G_n$, therefore we refer to the cases of $g_{1,n-1}(c_1,1)$ and $g_{n-1,1}(c_1,1)$ as $\textbf{maximal parabolics}$; and the cases of $g_{1,\cdots,1}(c_1,\cdots,c_n)$ as $\textbf{minimal parabolics}$.

\ 

We first recall the notion of principal series representations of general linear groups over finite fields. Let $\underline{\chi}_n$ be a character of the Borel subgroup $B_{n}^{-}$ defined by 
$$
\underline{\chi}_n\begin{pmatrix}
 a_1 &0 &\cdots &0 \\
    * &a_{2} &\ddots &\vdots \\
  \vdots & \ddots & \ddots &0\\
    * &\cdots &* & a_{n}\\

\end{pmatrix} = \prod_{i=1}^{n}\chi_i(a_i),
$$

where $\chi_i$, $1\leq i\leq n$, are multiplicative characters of $\IF_{q}$. A principal series representation ($\pi_{\underline{\chi}_n}, \CV_{\underline{\chi}_n})$ of $G_n$ with parameter $\underline{\chi}_n=(\chi_1,\cdots,\chi_n)$ is given by
$$
\CV_{\underline{\chi}_n}=Ind_{B_{n}^{-}}^{G_n}\underline{\chi}_n=\{f:G_n \rightarrow \IC \mid f(bg)=\underline{\chi}_n(b)f(g) \quad \forall b\in B_{n}^{-}, \forall g\in G_n\}.
$$

\ 

From (1.2) we see the ($B_{n}^{-},B_{n}^{-}$)-double cosets in $G_n$ are parametrized by $W_n$.

For $f_1, f_2\in \CV_{\underline{\chi}_n}$, we define the following inner product on $\CV_{\underline{\chi}_n}$:
\begin{align}
    <f_1, f_2>=\frac{1}{|G_n|}\sum_{g\in G_n}f_1(g)\overline{f_2(g)}.
\end{align}

By (1.1), we may rewrite (1.3) as follows:
\begin{align}
    <f_1, f_2>=\frac{1}{[G_n:B_{n}^{-}]}\sum_{u\in U_n}\sum_{w\in W_n}f_1(wu)\overline{f_2(wu)}.
\end{align}

\ 

Consider $\psi_{n}\in \CV_{\underline{\chi}_n}$ defined by the following:
\begin{align}
    \psi_{n}(wu)=\begin{cases} \psi(u), \quad \quad w=I_n\\ 0, \quad \quad \text{otherwise}\end{cases}.
\end{align}
$\psi_{n}$ is called the ($\psi$-)$\textbf{Whittaker vector}$ of $\CV_{\underline{\chi}_n}$. We have the following observation on $\psi_{n}$:
\begin{lem}
For every $u' \in U_n$, we have $\pi_{\underline{\chi}_n}(u')\psi_n=\psi(u')\psi_n$.
\end{lem}
\begin{proof}
For any given $u' \in U_n$, we have $\pi_{\underline{\chi}_n}(u')\psi_n(wu)=\psi_n(wuu')$, where $w\in W_n$ and $u\in U_n$ are fixed. If $w=I_n$, then clearly we have $\pi_{\underline{\chi}_n}(u')\psi_n(u)=\psi_n(uu')=\psi_n(u')\psi_n(u)$. If $w\neq I_n$, then $\psi_n(wuu')=0$ and $\psi_n(wu)=0$, therefore we still have $\pi_{\underline{\chi}_n}(u')\psi_n(wu)=\psi(u')\psi_n(wu)$ as both sides are 0.

Note that $\psi_n$ is determined by its values on $wu\in G_n$, therefore we conclude that $\pi_{\underline{\chi}_n}(u')\psi_n=\psi(u')\psi_n$.
\end{proof}
\

It is a standard result that ($\pi_{\underline{\chi}_n}, \CV_{\underline{\chi}_n})$ is an irreducible generic representation if $\chi_i$ ($1\leq i\leq n$) are all distinct. Kilmoyer \cite{kil} showed that $\CV_{\underline{\chi}_n}$ and $Ind_{U_n}^{G_n}\psi$ have exactly one irreducible constituent in common, denoted by ($\pi_{\underline{\chi}_n}^{S}, \CV_{\underline{\chi}_n}^{S}$) which is called generalized Steinberg representation in \cite{kil}. Moreover, $\psi_n \in\CV_{\underline{\chi}_n}^{S}$. We could identify $\CV_{\underline{\chi}_n}^{S}$ with its Whittaker model $\CW(\pi_{\underline{\chi}_n}^{S}, \psi)$ via certain matrix coefficients of the representation ($\pi_{\underline{\chi}_n}^{S}, \CV_{\underline{\chi}_n}^{S}$):
\begin{de}
For any $f\in \CV_{\underline{\chi}_n}$, we have the following matrix coefficient of ($\pi_{\underline{\chi}_n}, \CV_{\underline{\chi}_n})$
\begin{align}
    W_{\psi}(f): G_n\rightarrow \IC, \quad  g\mapsto <\pi_{\underline{\chi}_n}(g)f,\psi_n>.
\end{align}
Moreover, by (1.3)-(1.5) we have
\begin{align}
    W_{\psi}(f)(g)=\frac{1}{[G_n:B_{n}^{-}]}\sum_{u\in U_n}(\pi_{\underline{\chi}_n}(g)f)(u)\overline{\psi_n(u)}.
\end{align}
In particular, we denote $W_{\psi}(\psi_n)$ by $\Psi_{\underline{\chi}_n}^{GL_n}$, which is called the $\textbf{Whittaker function}$ associated with principal series representation ($\pi_{\underline{\chi}_n}, \CV_{\underline{\chi}_n})$. 
\end{de}
Note that $W_{\psi}(f)$ is a matrix coefficient of ($\pi_{\underline{\chi}_n}^{S}, \CV_{\underline{\chi}_n}^{S}$) if $f\in \CV_{\underline{\chi}_n}^{S}$, in particular, $W_{\psi}(\psi_n)$ is a matrix coefficient of both ($\pi_{\underline{\chi}_n}, \CV_{\underline{\chi}_n})$ and ($\pi_{\underline{\chi}_n}^{S}, \CV_{\underline{\chi}_n}^{S})$. We have the following explicit identification of the generalized Steinberg representation with its Whittaker model:
\begin{prop}
We have the following identification of ($\pi_{\underline{\chi}_n}^{S}, \CV_{\underline{\chi}_n}^{S})$ with its Whittaker model $\CW(\pi_{\underline{\chi}_n}^{S}, \psi)$:
$$
W_{\psi}: \CV_{\underline{\chi}_n}^{S}\rightarrow \CW(\pi_{\underline{\chi}_n}^{S}, \psi), \quad\quad f\mapsto W_{\psi}(f),
$$
where for any $g\in G_n$, 
$$
W_{\psi}(f)(g)=<\pi_{\underline{\chi}_n}(g)f,\psi_n>=\frac{1}{[G_n:B_{n}^{-}]}\sum_{u\in U_n}(\pi_{\underline{\chi}_n}(g)f)(u)\overline{\psi_n(u)}.
$$
\end{prop}
\begin{proof}
For any given $u\in U_n$ and $g\in G_n$, we have
$$
W_{\psi}(f)(ug)=<\pi_{\underline{\chi}_n}(ug)f,\psi_n>=<\pi_{\underline{\chi}_n}(g)f,\pi_{\underline{\chi}_n}(u^{-1})\psi_n>.
$$
By Lemma 1.1 we have
$$
\pi_{\underline{\chi}_n}(u^{-1})\psi_n=\psi_n(u^{-1})\psi=\overline{\psi(u)}\psi.
$$
Therefore 
$$
W_{\psi}(f)(ug)=<\pi_{\underline{\chi}_n}(g)f,\overline{\psi(u)}\psi_n>=\psi(u)<\pi_{\underline{\chi}_n}(g)f,\psi_n>.
$$
Hence we deduce that
$$
W_{\psi}(f)(ug)=\psi(u)W_{\psi}(f)(ug),
$$
whence for all $f\in \CV_{\underline{\chi}_n}^{S}$, $W_{\psi}(f)\in Ind_{U_n}^{G_n}\psi$
It is clear that $W_{\psi}$ is a linear homomorphism between the underlying spaces. For any given $g, g'\in G_n$ we have
\begin{equation*}
    \begin{aligned}
    W_{\psi}(\pi_{\underline{\chi}_n}(g')f)(g)&=\frac{1}{[G_n:B_{n}^{-}]}\sum_{u\in U_n}\left(\pi_{\underline{\chi}_n}(g)(\underline{\pi}_n(g')f)\right)(u)\overline{\psi_n(u)} \\
    &= \frac{1}{[G_n:B_{n}^{-}]}\sum_{u\in U_n}\left(\pi_{\underline{\chi}_n}(gg')f)\right)(u)\overline{\psi_n(u)} \\
    &= <\pi_{\underline{\chi}_n}(gg')f,\psi_n> \\
    &=g' \cdot W_{\psi}(f)(g).
    \end{aligned}
\end{equation*}
Hence we deduce that $W_{\psi}$ embeds ($\pi_{\underline{\chi}_n}^{S}, \CV_{\underline{\chi}_n}^{S})$ into the Gelfand-Graev representation. 
Remark that 
\begin{align}
    W_{\psi}(\psi_n)(I_n)=\frac{1}{[G_n:B_{n}^{-}]}\sum_{u\in U_n}\psi_n(u)\overline{\psi_n(u)}=\frac{|U_n|}{[G_n:B_{n}^{-}]}.
\end{align}
Therefore $W_{\psi}\neq 0$. Hence by Schur's lemma we conclude that $W_{\psi}$ is an isomorphism of representations.
\end{proof}

\ 
One key observation in this note is that for subrepresentations principal series representations entering the Gelfand-Graev representation, the Bessel function defined in Definition 1.1 and the Whittaker function defined in Definition 1.2 are proportional to each other. Namely we have the following proposition.
\begin{prop} For any given ($\pi_{\underline{\chi}_n}^{S}, \CV_{\underline{\chi}_n}^{S}$) we have
\begin{align}
    J_{\pi_{\underline{\chi}_n}^{S}}=\frac{[G_n:B_{n}^{-}]}{|U_n|}\Psi_{\underline{\chi}_n}^{GL_n}.
\end{align}
\end{prop}
\begin{proof}
For any given $u_1, u_2\in U_n$ and $g\in G_n$, we have
$$
\Psi_{\underline{\chi}_n}^{GL_n}(u_1gu_2)=<\pi_{\underline{\chi}_n}(u_1gu_2)\psi_n,\psi_n>=<\pi_{\underline{\chi}_n}(g)\left(\pi_{\underline{\chi}_n}(u_2)\psi_n\right),\pi_{\underline{\chi}_n}(u_1^{-1})\psi_n>.
$$
By Lemma 1.1 we have
$$
\pi_{\underline{\chi}_n}(u_2)\psi_n=\psi(u_2)\psi,
$$
and
$$
\pi_{\underline{\chi}_n}(u_1^{-1})\psi_n=\overline{\psi(u_1)}\psi.
$$
Therefore 
$$
\Psi_{\underline{\chi}_n}^{GL_n}(u_1gu_2)=<\pi_{\underline{\chi}_n}(g)\left(\psi(u_2)\psi_n\right),\overline{\psi(u_2)}\psi_n>=\psi(u_1u_2)<\pi_{\underline{\chi}_n}(g)\psi_n,\psi_n>.
$$
Hence $\Psi_{\underline{\chi}_n}^{GL_n}$ is $U_n$-bi-equivariant. By (1.8) we have
$$
\frac{[G_n:B_{n}^{-}]}{|U_n|}\Psi_{\underline{\chi}_n}^{GL_n}(I_n)=1.
$$
Therefore the function $\frac{[G_n:B_{n}^{-}]}{|U_n|}\Psi_{\underline{\chi}_n}^{GL_n}$ satisfies the two properties stated in Definition 1.1. Hence by the uniqueness of Bessel function we have
$$
 J_{\pi_{\underline{\chi}_n}^{S}}=\frac{[G_n:B_{n}^{-}]}{|U_n|}\Psi_{\underline{\chi}_n}^{GL_n}.
$$
\end{proof}

\ 

From Proposition 1.2 we conclude that for every irreducible generic representations which appear as subrepresentations of principal series representations we may replace their Bessel functions by the Whittaker functions, which is an important observation while deriving the main results stated in the next section.
\section{Curtis homomorphisms and Gelfand-Tsetlin diagrams}

Another way to describe all irreducible generic representations is to consider the endomorphism algebra of $Ind_{U_n}^{G_n}\psi$, $End_{G_n}(Ind_{U_n}^{G_n}\psi)$, which is called the \textbf{unipotent Hecke algebra}. Steinberg \cite{St} showed that this unipotent Hecke algebra is commutative, which is equivalent to the multiplicity-free property of the representation $Ind_{U_n}^{G_n}\psi$. By Mackey's theorem, $End_{G_n}(Ind_{U_n}^{G_n}\psi)$ is isomorphic to the $U_n$-bi-equivariant functions on $G_n$, namely
$$
\CH_{\psi}(G_n,U_n)=\{f:G_n \rightarrow \IC \mid f(u_1gu_2)=\psi(u_1u_2)f(g) \quad \forall u_1, u_2\in U_n, \forall g\in G_n\},
$$
where the multiplicative structure of $\CH_{\psi}(G_n,U_n)$ is convolution of functions. Clearly we see that $\{J_{\pi} \mid \text{$\pi$ is an irreducible generic representation of $G_n$}\}$ is a basis of $\CH_{\psi}(G_n,U_n)$. Therefore the study of unipotent Hecke algebras is closely related to constructions of Bessel functions.

\ 

Let $e_{\psi_n}\IC[G_n]e_{\psi_n}$ be a subalgebra of the group algebra $\IC[G_n]$, where $e_{\psi_n}$ is an idempotent element in $\IC[G_n]$ given by
$$
e_{\psi_n}=\frac{1}{|U_n|}\sum_{u\in U_n}\overline{\psi(u)}u.
$$
It can be shown by Theorem 1.1 that $e_{\psi_n}g_{n_1,\cdots, n_s}(c_1,\cdots,c_s)e_{\psi_n}$ is a basis for $e_{\psi_n}\IC[G_n]e_{\psi_n}$. Moreover, from Section 3 of \cite{c1}, we can identify $e_{\psi_n}\IC[G_n]e_{\psi_n}$ with the unipotent Hecke algebra $\CH_{\psi}(G_n,U_n)$ via the following
\begin{align}
    e_{\psi_n}\IC[G_n]e_{\psi_n} \tilde{\rightarrow} \CH_{\psi}(G_n,U_n),\quad  e_{\psi_n}g_{n_1,\cdots, n_s}(c_1,\cdots,c_s)e_{\psi_n}\mapsto |U_n|^2C_{g_{n_1,\cdots, n_s}(c_1,\cdots,c_s)^{-1}},
\end{align}
where $C_{g_{n_1,\cdots, n_s}(c_1,\cdots,c_s)^{-1}}$ is the characteristic function on the double coset
$U_n g_{n_1,\cdots, n_s}(c_1,\cdots,c_s)^{-1} U_n$. One consequence of the above identification is that each Bessel function $J_{\pi}$ induces an algebra homomorphism from $\CH_{\psi}(G_n,U_n)$ to $\IC$, denoted by $\CJ_{\pi}$:
\begin{align}
    \CJ_{\pi}:e_{\psi_n}\IC[G_n]e_{\psi_n}\rightarrow \IC, \quad\quad e_{\psi_n}g_{n_1,\cdots, n_s}(c_1,\cdots,c_s)e_{\psi_n} \mapsto J_{\pi}(g_{n_1,\cdots, n_s}(c_1,\cdots,c_s)).
\end{align}

\ 

In his 1993 paper, Curtis (Theorem 4.2 in \cite{c1}) showed that for each each algebra homomorphism $f_{T,\theta}:e_{\psi_n}\IC[G_n]e_{\psi_n}\rightarrow \IC$, there exists a unique homomorphism of algebras $f_{T}:e_{\psi_n}\IC[G_n]e_{\psi_n}\rightarrow \IC[T], c\mapsto \sum\limits_{t\in T}f_{T}(c)(t)t$ (the \textbf{Curtis homomorphism}) independent of $\theta$, which makes the following diagram commutative:
$$
\begin{tikzcd}
e_{\psi_n}\IC[G_n]e_{\psi_n} \arrow{r}{f_{T,\theta}} \arrow{dr}{f_T} & \IC\\
&\IC[T] \arrow{u}{\widehat{\theta}}
\end{tikzcd}
$$
where $\widehat{\theta}$ is the homomorphism on group algebra $\IC[T]$ linearly extended from the character $\theta$ of the maximal torus $T$. Therefore we obtain
\begin{align}
    f_{T,\theta}(e_{\psi_n}g_{n_1,\cdots, n_s}(c_1,\cdots,c_s)e_{\psi_n})=\sum_{t\in T}f_{T}(e_{\psi_n}g_{n_1,\cdots, n_s}(c_1,\cdots,c_s)e_{\psi_n})(t)\chi(t).
\end{align}

\ 

We note that in \cite{c1} Curtis expressed $f_{T}$ in terms of Green functions over finite fields using Deligne-Lusztig theory \cite{dl}. Later with Shinoda \cite{cs}, by using the Curtis' formula obtained, they gave explicit computation for values of $f_T(g_{1,n-1}(a,1))$, where $a\in \IF_{q}^{\times}$. However their method is hard to generalize to arbitrary $g_{n_1,\cdots, n_s}(c_1,\cdots,c_s)$. 

\ 

Suppose $T=T_n$ is the maximal split torus, i.e., the subgroup of diagonal matrices, and $\theta=\underline{\chi}_n$ is a fixed character on group $T_n$, therefore by (2.2) the homomorphism $f_{T_n, \underline{\chi}_n}$ can be induced from the Bessel function $J_{\pi_{\underline{\chi}_n}^{S}}$ associated with principal series representation, namely
\begin{align}
    f_{T_n, \underline{\chi}_n}(e_{\psi_n}g_{n_1,\cdots, n_s}(c_1,\cdots,c_s)e_{\psi_n})=J_{\pi_{\underline{\chi}_n}^{S}}(g_{n_1,\cdots, n_s}(c_1,\cdots,c_s)).
\end{align}
Combining with Theorem 1.3 we obtain
\begin{align}
    f_{T_n, \underline{\chi}_n}(e_{\psi_n}g_{n_1,\cdots, n_s}(c_1,\cdots,c_s)e_{\psi_n})=\frac{[G_n:B_{n}^{-}]}{|U_n|}\Psi_{\underline{\chi}_n}^{GL_n}(g_{n_1,\cdots, n_s}(c_1,\cdots,c_s)).
\end{align}
Now by the recursive formulas obtained in Section 3 and 4, we are able to give an explicit formula for $f_{T_n}(g)(t)$, where $g$ is either in maximal parabolics or minimal parabolics, i.e., $g$ is either $g_{n-1,1}(1,a_n)$, $g_{1,n-1}(a_1,1)$, or $g_{1,\cdots, 1}(a_1,\cdots,a_n)$.

\ 
\begin{te}
For any $n\geq 2$ we have 
\begin{align}
    f_{T_n}(e_{\psi_n}g_{n-1,1}(1,a_n)e_{\psi_n})(t)=\begin{cases} \frac{1}{q^{n-1}}\psi(Tr(t^{-1})), \quad \quad det(t)=(-1)^{n-1}a_n\\ 0, \quad \quad \text{otherwise}\end{cases}.
\end{align}
Explicitly, we can parameterize $t\in T_n$ with $det(t)=(-1)^{n-1}a_n$ as follows:
$$
t=diag\{y_{1.1}, -\frac{y_{2,2}}{y_{1,1}},-\frac{y_{3,3}}{y_{2,2}},\cdots, -\frac{y_{n-1,n-1}}{y_{n-2,n-2}},-\frac{a_n}{y_{n-1,n-1}}\}.
$$
Therefore 
$$
\psi(Tr(t^{-1}))=\psi\left(\frac{1}{y_{1,1}}-\sum_{j=2}^{n-1}\frac{y_{j-1,j-1}}{y_{j,j}}-\frac{y_{n-1,n-1}}{a_n}\right),
$$
which is a summation over the arrows of the following Gelfand-Tsetlin diagram:
$$
\begin{tikzcd}
a_n &y_{n-1,n-1} \arrow[l]& y_{n-2,n-2} \arrow[l]& \cdots \arrow[l] &y_{1,1} \arrow[l]\\
  &&&& 1 \arrow[u],
\end{tikzcd}
 $$
 in the sense that each arrow $x\rightarrow y$ corresponds to $\psi(\pm \frac{x}{y})$, where the "+" sign corresponds to vertical arrow and "-" sign corresponds to horizontal arrow.
\end{te}

\begin{proof}
For each $n\geq 2$ we have the following
\begin{align}
    \Psi_{\underline{\chi}_{n}}^{GL_{n}}(g_{n-1,1}(1,a_n))=F'_n\sum_{\substack{y_{j,j}\in \IF_{q}^{\times} \\ 0\leq j\leq n}}\prod_{j=1}^{n} \chi_j\left(-\frac{y_{j,j}}{y_{j-1,j-1}}\right)\psi\left(-\sum_{j=1}^n\frac{y_{j-1,j-1}}{y_{j,j}}\right),
\end{align}
where $F'_n=\frac{|U_{n-1}|}{[G_n:B_{n}^{-}]}$ and we assume that $y_{n,n}=a_n$ and $y_{0,0}=-1$. Since for $n=2$ equation (2.7) is simply equation (3.6), and for $n\geq 3$ equation (2.7) is obtained by resolving the recursive relation given in Theorem 3.1.
Let $d_j=-\frac{y_{j,j}}{y_{j-1,j-1}}$ we obtain
\begin{align}
  \frac{[G_n:B_{n}^{-}]}{|U_n|} \Psi_{\underline{\chi}_{n}}^{GL_{n}}(g_{n-1,1}(1,a_n))&=\frac{|U_{n-1}|}{|U_n|}\sum_{\substack{y_{j,j}\in \mathbb{F}_{q}^{\times} \\ 0\leq j\leq n}}\prod_{j=1}^{n}\chi_{j}\left(d_j\right)\psi\left(\sum_{j=1}^n\frac{1}{d_j}\right)\\
  &=\frac{1}{q^{n-1}}\sum_{\substack{y_{j,j}\in \mathbb{F}_{q}^{\times} \\ 0\leq j\leq n}}\prod_{j=1}^{n}\chi_{j}\left(d_j\right)\psi(Tr(\text{diag}_{1\leq j\leq n}(d_j)^{-1})).
\end{align}
By equation (2.3) and (2.5) we have
\begin{align}
    \frac{[G_n:B_{n}^{-}]}{|U_n|} \Psi_{\underline{\chi}_{n}}^{GL_{n}}(g_{n-1,1}(1,a_n))=\sum_{t\in T_n}f_{T_n}(e_{\psi_n}g_{n-1,1}(1,a_n)e_{\psi_n})(t)\underline{\chi}_n(t).
\end{align}
Hence Theorem 2.1 follows by comparing equation (2.9) and (2.10).
 
\end{proof}
\ 

Similarly we have the following theorem for the case $g=g_{1,n-1}(a_1,1)$.
\begin{te}
For each $n\geq 2$ we have
\begin{align}
    f_{T_n}(e_{\psi_n}g_{1,n-1}(a_1,1)e_{\psi_n})(t)=\begin{cases} \frac{1}{q^{n-1}}\psi(-Tr(t)), \quad \quad det(t)=(-1)^{n-1}a_1\\ 0, \quad \quad \text{otherwise}\end{cases},
\end{align}
Explicitly, we can parameterize $t\in T_n$ with $det(t)=(-1)^{n-1}a_1$ as follows:
$$
t=diag\{\frac{y_{1.1}}{1}, -\frac{y_{2,1}}{y_{1,1}},-\frac{y_{3,1}}{y_{2,1}},\cdots, -\frac{y_{n-1,1}}{y_{n-2,1}},-\frac{a_1}{y_{n-1,1}}\}.
$$
Therefore 
$$
\psi(-Tr(t))=\psi\left(-\frac{y_{1,1}}{1}+\sum_{j=2}^{n-1}\frac{y_{j,1}}{y_{j-1,1}}+\frac{a_1}{y_{n-1,1}}\right),
$$
which is a summation over the arrows of the following Gelfand-Tsetlin diagram:
$$
\begin{tikzcd}
1 &y_{1,1} \arrow[l]\\
   &y_{2,1} \arrow[u]   \\
&\vdots \arrow[u]\\
 & y_{n-2,1}\arrow[u] \\
& y_{n-1,1}\arrow[u] \\
&a_1\arrow[u]
\end{tikzcd}
 $$
  In the sense that each arrow $x\rightarrow y$ corresponds to $\psi(\pm \frac{x}{y})$, where the "+" sign corresponds to vertical arrow and "-" sign corresponds to horizontal arrow.
\end{te}

\begin{proof}
For each $n\geq 2$ we have the following
\begin{align}
    \Psi_{\underline{\chi}_{n}}^{GL_{n}}(g_{1,n-1}(a_{n-1},1))=F'_n\sum_{\substack{y_{j,j}\in \IF_{q}^{\times} \\ 0\leq j\leq n}}\prod_{j=1}^{n} \chi_j\left(-\frac{y_{j,1}}{y_{j-1,1}}\right)\psi\left(\sum_{j=1}^n\frac{y_{j,1}}{y_{j-1,1}}\right),
\end{align}
where $F'_n=\frac{|U_{n-1}|}{[G_n:B_{n}^{-}]}$ and we assume that $y_{n,1}=a_1$ and $y_{0,1}=-1$. Since for $n=2$ equation (2.12) is simply equation (3.6), and for $n\geq 3$ equation (2.12) is obtained by resolving the recursive relation given Theorem 3.2.
Let $d_j=-\frac{y_{j,1}}{y_{j-1,1}}$ we obtain
\begin{align}
  \frac{[G_n:B_{n}^{-}]}{|U_n|} \Psi_{\underline{\chi}_{n}}^{GL_{n}}(g_{n-1,1}(1,a_n))&=\frac{|U_{n-1}|}{|U_n|}\sum_{\substack{y_{j,j}\in \mathbb{F}_{q}^{\times} \\ 0\leq j\leq n}}\prod_{j=1}^{n}\chi_{j}\left(d_j\right)\psi\left(-\sum_{j=1}^nd_j\right)\\
  &=\frac{1}{q^{n-1}}\sum_{\substack{y_{j,j}\in \mathbb{F}_{q}^{\times} \\ 0\leq j\leq n}}\prod_{j=1}^{n}\chi_{j}\left(d_j\right)\psi(-Tr(\text{diag}_{1\leq j\leq n}(d_j))).
\end{align}
By equation (2.3) and (2.5) we have
\begin{align}
    \frac{[G_n:B_{n}^{-}]}{|U_n|} \Psi_{\underline{\chi}_{n}}^{GL_{n}}(g_{1,n-1}(a_1,1))=\sum_{t\in T_n}f_{T_n}(e_{\psi_n}g_{1,n-1}(a_1,1)e_{\psi_n})(t)\underline{\chi}_n(t).
\end{align}
Hence Theorem 2.2 follows by comparing equation (2.14) and (2.15).
 
\end{proof}

\begin{rem}
In Theorem 2.2 we give an alternative proof of Theorem 3.3 in \cite{cs} without using Deligne-Lusztig theory. Indeed, what Curtis and Shinoda have shown is that
\begin{align}
    f_{T_n}(q^{n-1}(g_{1,n-1}(a_1,1)))(t)=\begin{cases} \psi(-Tr(t)), \quad \quad det(t)=(-1)^{n-1}a_1\\ 0, \quad \quad \text{otherwise}\end{cases},
\end{align}
and the constant $q^{n-1}$ equals to $ind(g_{1,n-1}(a_1,1))=\frac{|U_ng_{1,n-1}(a_1,1)U_n|}{|U_n|}$.
Our argument illustrates that $ind(g_{1,n-1}(a_1,1))$ is determined by the recursive formulas of Whittaker functions if the homomorphism $f_{T_n, \underline{\chi}_n}$ was chosen proproperly.
\end{rem}

\ 

The same argument applies when $g=g_{1,\cdots, 1}(a_1,\cdots,a_n)$. Note that we have

\begin{te}
For each $n\geq 2$, $f_{T_n}(g_{1,\cdots, 1}(a_1,\cdots,a_n))(t)\neq 0$ if and only if  $det(t)=(-1)^{\frac{n(n-1)}{2}}a_1\cdots a_n$, where we could parameterize $t$ by $\text{diag}_{1\leq j\leq n}(d_j)$ with
$$
d_j=(-1)^{j-1}y_{j,j}\prod_{k=1}^{j-1}\frac{y_{j,k}}{y_{j-1,k}},
$$
where for all $k$, $y_{n,k}=a_k$, $y_{0,k}=1$. Therefore we have
\begin{align}
    f_{T_n}(e_{\psi_n}g_{1,\cdots, 1}(a_1,\cdots,a_n)e_{\psi_n})(\text{diag}_{1\leq j\leq n}(d_j)) \nonumber \\
    =\frac{1}{|U_n|}\psi\left(\sum_{i=2}^n\sum_{k=1}^{i-1}\left(\frac{y_{i,k}}{y_{i-1,k}}-\frac{y_{i-1,k}}{y_{i,k+1}}\right)\right),
\end{align}
which is a summation over the arrows of the following Gelfand-Tsetlin diagram
$$
\begin{tikzcd}
a_n &y_{n-1,n-1}\arrow[l]  &y_{n-2,n-2}\arrow[l]&\cdots \arrow[l] &y_{2,2} \arrow[l] &y_{1,1} \arrow[l]\\
  & a_{n-1} \arrow[u]&y_{n-1,n-2}\arrow[u]  \arrow[l]  & \cdots \arrow[l] &y_{3,2} \arrow[u] \arrow[l]  &y_{2,1} \arrow[u]  \arrow[l] \\
   &&\ddots &\ddots &\vdots \arrow[u]&\vdots \arrow[u]\\
   && &a_3\arrow[u] &y_{n-1,2} \arrow[u] \arrow[l] & y_{n-2,1}\arrow[u] \arrow[l]\\
   &&&& a_2\arrow[u]& y_{n-1,1}\arrow[u] \arrow[l]\\
   &&&&&a_1\arrow[u]
\end{tikzcd}
$$
In the sense that each arrow $x\rightarrow y$ corresponds to $\psi(\pm \frac{x}{y})$, where the "+" sign corresponds to vertical arrow and "-" sign corresponds to horizontal arrow.
\end{te}

\begin{proof}
Resolving the recursive relations for minimal parabolics given in Theorem 4.1 we obtain
\begin{align}
     \Psi_{\underline{\chi}_{n}}^{GL_{n}}(g_{1,\cdots, 1}(a_1,\cdots,a_n))=F_n\sum_{\substack{y_{j,k}\in \IF_{q}^{\times} \\ 0\leq j\leq n\\ 1\leq k\leq j}}&\prod_{j=1}^{n}  \chi_j\left((-1)^{j-1}y_{j,j} \prod_{k=1}^{j-1}\frac{y_{j,k}}{y_{j-1,k}}\right)\nonumber \\  &\times \prod_{j=2}^{n}\widetilde{Q}_{GL_{j-1}}^{GL_j}(y_{j,1},\cdots,y_{j,j};y_{j-1,1},\cdots, y_{j-1,j-1}|\chi_{j}),
\end{align}
where $F_n=\frac{1}{[G_n:B_{n}^{-}]}$ and we assume that for all $k$, $y_{n,k}=a_k$, $y_{0,k}=1$. 

Let $d_j=(-1)^{j-1}y_{j,j}\prod_{k=1}^{j-1}\frac{y_{j,k}}{y_{j-1,k}}$, we have
\begin{align}
    \frac{[G_n:B_{n}^{-}]}{|U_n|} \Psi_{\underline{\chi}_{n}}^{GL_{n}}(g_{1,\cdots, 1}(a_1,\cdots,a_n))=\sum_{\substack{y_{j,k}\in \IF_{q}^{\times} \\ 0\leq j\leq n\\ 1\leq k\leq j}}&\prod_{j=1}^{n}  \chi_j(d_j) \nonumber \\
    &\times\frac{1}{|U_n|}\prod_{j=2}^{n}\widetilde{Q}_{GL_{j-1}}^{GL_j}(y_{j,1},\cdots,y_{j,j};y_{j-1,1},\cdots, y_{j-1,j-1}|\chi_{j}).
\end{align}

By equation (2.3) and (2.5) we have
\begin{align}
    \frac{[G_n:B_{n}^{-}]}{|U_n|} \Psi_{\underline{\chi}_{n}}^{GL_{n}}(g_{1,\cdots, 1}(a_1,\cdots,a_n))=\sum_{t\in T_n}f_{T_n}(e_{\psi_n}g_{1,\cdots, 1}(a_1,\cdots,a_n)e_{\psi_n})(t)\underline{\chi}_n(t).
\end{align}
Hence by comparing (2.18) and (2.19) we obtain
\begin{align}
    f_{T_n}(e_{\psi_n}g_{1,\cdots, 1}(a_1,\cdots,a_n)e_{\psi_n})(\text{diag}_{1\leq j\leq n}(d_j)) \nonumber \\
    =\frac{1}{|U_n|}\prod_{j=2}^{n}\widetilde{Q}_{GL_{j-1}}^{GL_j}(y_{j,1},\cdots,y_{j,j};y_{j-1,1},\cdots, y_{j-1,j-1}|\chi_{j}).
\end{align}
Equation (2.17) is obtained by (2.21) and (4.3).
\end{proof}

\begin{rem}
We give explicit computation of Curtis homomorphism for both maximal and minimal parabolics in Theorem 2.1-2.3 and show that they are indeed summations over corresponding Gelfand-Tsetlin diagrams. Note that by equation (2.17) we have
\begin{align}
    ind(g_{1,\cdots, 1}(a_1,\cdots,a_n))=|U_n|.
\end{align}
\end{rem}

\section{Gelfand-Tsetlin formulas for maximal parabolics}

Suppose $n\geq 3$. Let $w_1=\left(\begin{array}{c|ccc}
  \begin{smallmatrix}
  0 \\ \vdots \\ 0
  \end{smallmatrix}
  & \begin{smallmatrix}
  1 & &0 \\
      &\ddots & \\
   0&    &1\\
\end{smallmatrix} \\
\hline 
  \begin{smallmatrix}
  1
  \end{smallmatrix} &
  \begin{smallmatrix}
  0&  \cdots& 0
  \end{smallmatrix}
\end{array}\right)$
, and $a=\left(\begin{smallmatrix}
  1& & & \\
   & \ddots &  &\\
     & &1 & \\
     &&& a_n
\end{smallmatrix}\right)$, we aim to compute $\Psi_{\underline{\chi}_n}^{GL_n}(aw_1)$, remark that 
$$
\Psi_{\underline{\chi}_n}^{GL_n}(aw_1)=<\pi_{\underline{\chi}_n}(aw_1)\psi_n,\psi_n>=<\pi_{\underline{\chi}_n}(w_1)\psi_n,\pi_{\underline{\chi}_n}(a^{-1})\psi_n>
$$
Let $u\in U_n$ be arbitrary, for any $1\leq i<j\leq n$ we denote the $(i,j)$ entry of $u$ by $z_{i,j}$, and we have the following parametrization of $u$:
$$
u=(I_n+y_{1,1}E_{1,2})(I_n+y_{2,1}E_{1,2}+y_{2,2}E_{2,3})\cdots(I_n +\sum_{i=1}^{k}y_{k,i}E_{i,i+1}).
$$
For simplicity we write the following short-hand notation for this parametrization for the rest of the note.
$$
u=\prod_{k=1}^{n-1}(I_n +\sum_{i=1}^{k}y_{k,i}E_{i,i+1}).
$$
By direct computation we have
$$
\pi_{\underline{\chi}_n}(a^{-1})\psi_n(u)=\chi_{n}(a_{n}^{-1})\psi\left(\sum_{i=1}^{n-2}z_{i,i+1}+\frac{z_{n-1,n}}{a_n}\right).
$$
Moreover,
$$
\pi_{\underline{\chi}_n}(w_1)\psi_n(u)=\prod_{i=1}^{n}\chi_i\left(\frac{\Delta_i(uw_1)}{\Delta_{i-1}(uw_1)}\right)\psi\left(\sum_{i=1}^{n-1}\frac{\Delta_{i,i+1}(uw_1)}{\Delta_i(uw_1)}\right),
$$
where $\Delta_{i}(uw_1)$ denotes the principal $i\times i$ minor of the matrix $uw_1$ and $\Delta_{i,i+1}$ denotes the determinant obtained from $\Delta_{i}(uw_1)$ via replacing the $i$-th and $i+1$-st column in $uw_1$, and we assume that $\Delta_0=1$.
Note that we have
$$
z_{i,n}=\prod_{k=i}^{n-1}y_{k,k},
$$
Therefore for all $i\geq 1$,
$$
\Delta_{i}(uw_1)=(-1)^{i+1}z_{i,n}=(-1)^{i+1}\prod_{k=i}^{n-1}y_{k,k}.
$$
For $i>1$ we have
$$
\Delta_{i,i+1}(uw_1)=(-1)^{i+1}z_{i,n}z_{i-1,i}+(-1)^{i}z_{i-1,n},
$$
and $\Delta_{1,2}=1$.
Let $F_n=[G_n:B_{n}^{-}]^{-1}$, we obtain
\begin{align}
    \Psi_{\underline{\chi}_n}^{GL_n}(aw_1)&=F_n\sum_{\substack{z_{i,j}\\ 1<i<j<n}}\prod_{i=1}^{n}\chi_i\left(\frac{\Delta_i(uw_1)}{\Delta_{i-1}(uw_1)}\right)\psi\left(\sum_{i=1}^{n-1}\frac{\Delta_{i,i+1}(uw_1)}{\Delta_i(uw_1)}\right)\chi_{n}(a_{n})\psi\left(-\sum_{i=1}^{n-2}z_{i,i+1}-\frac{z_{n-1,n}}{a_n}\right) \\
    &=F_n\sum_{\substack{z_{i,j}\\ 1<i<j<n}}\chi_{n}(a_{n})\prod_{i=1}^{n}\chi_i\left(\frac{\Delta_i}{\Delta_{i-1}}\right)\psi\left(\frac{1}{z_{1,n}}+\sum_{i=2}^{n-1}z_{i-1,i}-\sum_{i=2}^{n}\frac{z_{i-1,n}}{z_{i,n}} -\sum_{i=1}^{n-2}z_{i,i+1}-\frac{z_{n-1,n}}{a_n}\right)
\end{align}
By (3.2), $z_{i,i+1}=\sum\limits_{k=i}^{n}y_{k,i}$ has no contribution to the sum, therefore we have the following
\begin{align}
    \Psi_{\underline{\chi}_n}^{GL_n}(aw_1)=q^{\frac{(n-1)(n-2)}{2}}F_n &\sum_{\substack{y_{i,i}\\ 1\leq i\leq n-1}}\chi_1(y_{1,1}\cdots y_{n-1,n-1})\chi_{n}\left(-\frac{a_{n}}{y_{n-1,n-1}}\right)\prod_{i=2}^{n-1}\chi_i\left(-\frac{1}{y_{i-1,i-1}}\right)\\ \nonumber
    & \times \psi\left(\frac{1}{y_{1,1}\cdots y_{n-1,n-1}}-\sum_{i=2}^{n}y_{i-1,i-1} -\frac{y_{n-1,n-1}}{a_n}\right).
\end{align}
Let $F'_n=q^{\frac{(n-1)(n-2)}{2}}F_n$, therefore by (3.3) we have
\begin{align}
    \Psi_{\underline{\chi}_n}^{GL_n}(aw_1)=&C'_n \sum_{y_{n-1,n-1}} \chi_n\left(-\frac{a_n}{y_{n-1,n-1}}\right)\psi\left(-\frac{y_{n-1,n-1}}{a_n}\right) \nonumber \\
    & \times F'_{n-1} \sum_{\substack{y_{i,i}\\ 1\leq i\leq n-2}}\chi_1(y_{1,1}\cdots y_{n-1,n-1})\prod_{i=2}^{n-1}\chi_i\left(-\frac{1}{y_{i-1,i-1}}\right)
   \psi\left(\frac{1}{y_{1,1}\cdots y_{n-1,n-1}}-\sum_{i=2}^{n}y_{i-1,i-1} \right).
\end{align}
Consider the inner summation of (3.4). By changing variable $y_{n-2,n-2}\mapsto y_{n-1,n-1}y_{n-2,n-2}$, we have
\begin{align}
    F'_{n-1} \sum_{\substack{y_{i,i}\\ 1\leq i\leq n-2}} &\chi_1(y_{1,1}\cdots y_{n-2,n-2})\chi_{n-1}\left(-\frac{y_{n-1,n-1}}{y_{n-2,n-2}}\right)\prod_{i=2}^{n-2}\chi_i\left(-\frac{1}{y_{i-1,i-1}}\right) \\ \nonumber
   & \times \psi\left(\frac{1}{y_{1,1}\cdots y_{n-2,n-2}}-\sum_{i=2}^{n-1}y_{i-1,i-1}-\frac{y_{n-2,n-2}}{y_{n-1,n-1}} \right).
\end{align}

Clearly we have
\begin{align}
    \Psi_{\underline{\chi}_{2}}^{GL_{2}}\left(\begin{matrix}
    0&a_1 \\ a_2&0
    \end{matrix}\right)=\frac{1}{q+1}\sum_{y_{1,1}\in \IF_{q}^{\times}}\chi_2\left(-\frac{a_1a_2}{y_{1,1}}\right)\psi\left(-\frac{y_{1,1}}{a_2}+\frac{a_{1}}{y_{1,1}}\right)\chi_1(y_{1,1}).
\end{align}

Combining (3.4)-(3.6) we have the following theorem
\begin{te}(Gelfand-Tsetlin formula for maximal parabolics \RNum{1}) Let $n\geq 3$, then for every $a_1, t_1\in \IF_{q}^{\times}$ we have
\begin{align}
    \Psi_{\underline{\chi}_{n}}^{GL_{n}}(g_{n-1,1}(1,a_n))=C'_n\sum_{t_{n-1}\in \IF_{q}^{\times}} \chi_n\left(-\frac{a_n}{t_{n-1}}\right)\psi\left(-\frac{t_{n-1}}{a_n}\right)\Psi_{\underline{\chi}_{n-1}}^{GL_{n-1}}(g_{n-2,1}(1,t_{n-1})),
\end{align}
where $C'_n=q^{n-2}C_n=\frac{q^{n-2}(q-1)}{q^n-1}$.
For the combinatorial interpretation, set $t_{n-1}=y_{n-1,n-1}$, therefore we have the following reduced Gelfand-Tsetlin diagram
$$
\begin{tikzcd}
a_n &y_{n-1,n-1} \arrow[l]& y_{n-2,n-2} \arrow[l,dashed]& \cdots \arrow[l,dashed] &y_{1,1} \arrow[l,dashed]\\
  &&&& 1 \arrow[u,dashed]
\end{tikzcd}
 $$
\end{te}
We also have the following theorem "dual" to the above
\begin{te}(Gelfand-Tsetlin formula for maximal parabolics \RNum{2}) Let $n\geq 3$, then for every $a_1, t_1\in \IF_{q}^{\times}$ we have
\begin{align}
    \Psi_{\underline{\chi}_{n}}^{GL_{n}}(g_{1,n-1}(a_1,1))=C'_n\sum_{t_1\in \IF_{q}^{\times}} \chi_n\left(-\frac{a_1}{t_1}\right)\psi\left(\frac{a_1}{t_1}\right)\Psi_{\underline{\chi}_{n-1}}^{GL_{n-1}}(g_{1,n-2}(t_1,1)),
\end{align}
where $C'_n=q^{n-2}C_n=\frac{q^{n-2}(q-1)}{q^n-1}$.
For the combinatorial interpretation, set $t_1=y_{n-1,1}$, therefore we have the following reduced Gelfand-Tsetlin diagram
$$
\begin{tikzcd}
1 &y_{1,1} \arrow[l,dashed]\\
   &y_{2,1} \arrow[u,dashed]   \\
&\vdots \arrow[u,dashed]\\
 & y_{n-2,1}\arrow[u,dashed] \\
& y_{n-1,1}\arrow[u,dashed] \\
&a_1\arrow[u]
\end{tikzcd}
 $$
\end{te}
\proof We have 
\begin{align}
    \Psi_{\underline{\chi}_{n}}^{GL_{n}}(g_{1,n-1}(a_n^{-1},1))&=\overline{\Psi_{\underline{\chi}_{n}}^{GL_{n}}(g_{n-1,1}(1,a_n))} \\
    &= C'_n\sum_{t_{n-1}\in \IF_{q}^{\times}} \chi_n\left(-\frac{t_{n-1}}{a_n}\right)\psi\left(\frac{t_{n-1}}{a_n}\right)\overline{\Psi_{\underline{\chi}_{n-1}}^{GL_{n-1}}(g_{n-2,1}(1,t_{n-1}))}\\
    &= C'_n\sum_{t_{n-1}\in \IF_{q}^{\times}} \chi_n\left(-\frac{t_{n-1}}{a_n}\right)\psi\left(\frac{t_{n-1}}{a_n}\right)\Psi_{\underline{\chi}_{n-1}}^{GL_{n-1}}(g_{1,n-2}(t_{n-1}^{-1},1)).
\end{align}
Setting $a_1=a_{n}^{-1}$ and $t_1=t_{n-1}^{-1}$ in (3.11), we obtain (3.8).

\section{Gelfand-Tsetlin formula for minimal parabolics}

In this section we aim to give explicit formula for  $\Psi_{\underline{\chi}_n}^{GL_n}\left(\begin{smallmatrix}
   0 & &a_1 \\
     & \iddots &\\
a_n    & & 0
\end{smallmatrix}\right)$ in terms of $\Psi_{\underline{\chi}_{n-1}}^{GL_{n-1}}$, i.e., Whittaker function associated with principal series representation of $G_{n-1}$ with parameter $\underline{\chi}_{n-1}=(\chi_1,\cdots,\chi_{n-1})$:

\begin{te}(Gelfand-Tsetlin formula for minimal parabolics)
For all $n\geq 3$ ,let $a=diag(a_1,\cdots,a_n)$, $w_0=\left(\begin{smallmatrix}
   0 & &1 \\
     & \iddots &\\
1    & & 0
\end{smallmatrix}\right)$. Set $\Psi_{\underline{\chi}_{n-1}}^{GL_{n-1}}(aw_0)=\Psi_{\underline{\chi}_{n-1}}^{GL_{n-1}}(a_1,\cdots,a_n)$, we have the following recursion relation:
\begin{align}
    \Psi_{\underline{\chi}_{n}}^{GL_{n}}(a_1,\cdots,a_n)=C_n\sum_{t_1,\cdots, t_{n-1}\in \IF_{q}^{\times}} Q_{GL_{n-1}}^{GL_{n}}(a_1,\cdots,a_n;t_1,\cdots, t_{n-1}|\chi_n)\Psi_{\underline{\chi}_{n-1}}^{GL_{n-1}}(t_1,\cdots, t_{n-1}),
\end{align}
where $C_n=\frac{q-1}{q^n-1}$ and the recursion kernel is given by
\begin{align}
    &Q_{GL_{n-1}}^{GL_{n}}(a_1,\cdots,a_n;t_1,\cdots, t_{n-1}|\chi_n)\nonumber \\
    &=\chi_n\left((-1)^{n-1}\frac{a_1\cdots a_n}{t_1\cdots t_{n-1}}\right)\widetilde{Q}_{GL_{n-1}}^{GL_{n}}(a_1,\cdots,a_n;t_1,\cdots, t_{n-1}|\chi_n)
\end{align}
where
\begin{align}
    \widetilde{Q}_{GL_{n-1}}^{GL_{n}}(a_1,\cdots,a_n;t_1,\cdots, t_{n-1}|\chi_n)=\psi\left(\sum_{k=1}^{n-1}\frac{a_k}{t_{k}}-\sum_{k=1}^{n-1}\frac{t_k}{a_{k+1}}\right).
\end{align}
\end{te}

\ 

We introduce the following combinatorial interpretation of the recursion formula before proving it. Recall that every element $u\in U_n$ can be parameterized by the following
$$
u=\prod_{k=1}^{n-1}(I_n +\sum_{i=1}^{k}y_{k,i}E_{i,i+1}),
$$
where $E_{i,i+1}$ is the elementary $n\times n$ matrix with 1 at the $(i,i+1)$ entry and zeros otherwise. The recursion relation (4.1) corresponds to the following Gelfand-Tsetlin diagram
$$
\begin{tikzcd}
a_n &y_{n-1,n-1}\arrow[l]  &y_{n-2,n-2}\arrow[l,dashed]&\cdots \arrow[l,dashed] &y_{2,2} \arrow[l,dashed] &y_{1,1} \arrow[l,dashed]\\
  & a_{n-1} \arrow[u]&y_{n-1,n-2}\arrow[u,dashed]  \arrow[l]  & \cdots \arrow[l,dashed] &y_{3,2} \arrow[u,dashed] \arrow[l,dashed]  &y_{2,1} \arrow[u,dashed]  \arrow[l,dashed] \\
   &&\ddots &\ddots &\vdots \arrow[u,dashed]&\vdots \arrow[u,dashed]\\
   && &a_3\arrow[u] &y_{n-1,2} \arrow[u,dashed] \arrow[l] & y_{n-2,1}\arrow[u,dashed] \arrow[l,dashed]\\
   &&&& a_2\arrow[u]& y_{n-1,1}\arrow[u,dashed] \arrow[l]\\
   &&&&&a_1\arrow[u]
\end{tikzcd}
 $$
In the above diagram $x\rightarrow y$ corresponds to $\psi(\pm \frac{x}{y})$, in the formula (4.1), where the "+" sign corresponds to vertical arrows and "-" sign corresponds to horizontal arrows. Clearly the solid arrows correspond to terms in (4.3). Let $t_k=y_{n-1,k}$ for $1\leq k\leq n-1$, then the dashed arrows would be solid if we expand $\Psi_{\underline{\chi}_{n-1}}^{GL_{n-1}}(t_1,\cdots, t_{n-1})$ using the recursion relation repeatedly.

\ 

The following result is the key step of the proof of Theorem 4.1.

\begin{prop}
For every $1\leq i\leq n-1$, where $n\geq 2$, we have the following formulas regarding $\Delta_i=\Delta_i(uaw_0)$ and $\Delta_{i,i+1}=\Delta_{i,i+1}(uaw_0)$:
\begin{align}
    \Delta_i=(-1)^{\frac{i(i-1)}{2}}\prod_{k=i}^{n-1}\prod_{j=k-i+1}^{k}y_{k,j}\prod_{k=1}^{i}a_{n-k+1},
\end{align}
and
\begin{align}
    \frac{\Delta_{i,i+1}}{\Delta_i}=\frac{a_{n-i}}{a_{n-i+1}}\sum_{k=i}^{n-1}\prod_{m=k}^{n-1}\frac{y_{m+1,m-i+1}}{y_{m,m-i+1}},
\end{align}
where $y_{n,k}=1$ is assumed for any $1\leq k\leq n-1$.
\end{prop}

We will prove this proposition in the last section. Indeed Theorem 4.1 easily follows from the above proposition.
\begin{proof}(Theorem 4.1)
We have
\begin{align}
    \pi_{\underline{\chi}_n}(aw_0)\psi_n(u)&=\prod_{i=1}^{n}\chi_i\left(\frac{\Delta_i}{\Delta_{i-1}}\right)\psi\left(\sum_{i=1}^{n-1}\frac{\Delta_{i,i+1}}{\Delta_i}\right) \\
    &=\chi_1(a_n\prod_{k=i}^{n-1}y_{k,k})\prod_{i=2}^{n}\chi_i\left((-1)^{i-1}\frac{\prod_{k=i}^{n-1}y_{k,k-i+1}a_{n-i+1}}{\prod_{j=1}^{i-1}y_{i-1,j}}\right)\nonumber \\
    &\times \psi\left(\sum_{i=1}^{n-1}\frac{a_{n-i}}{a_{n-i+1}}\sum_{k=i}^{n-1}\prod_{m=k}^{n-1}\frac{y_{m+1,m-i+1}}{y_{m,m-i+1}}\right)
\end{align}
Therefore 
\begin{align}
    \Psi_{\underline{\chi}_{n}}^{GL_{n}}(a_1,\cdots,a_n)=&F_n\sum_{\substack{y_{k,i}\in \mathbb{F}_{q}^{\times} \\ 1\leq k\leq n-1\\ 1\leq i\leq k}}\chi_1(a_n\prod_{k=i}^{n-1}y_{k,k})\prod_{i=2}^{n}\chi_i\left((-1)^{i-1}\frac{\prod_{k=i}^{n-1}y_{k,k-i+1}a_{n-i+1}}{\prod_{j=1}^{i-1}y_{i-1,j}}\right)\nonumber \\
    &\times \psi\left(\sum_{i=1}^{n-1}\frac{a_{n-i}}{a_{n-i+1}}\sum_{k=i}^{n-1}\prod_{m=k}^{n-1}\frac{y_{m+1,m-i+1}}{y_{m,m-i+1}}\right)\psi\left(-\sum_{k=1}^{n-1}\sum_{i=1}^{k}y_{k,i}\right)
\end{align}
Let $y'_{n-1,k}= y_{n-1,k}a_{k+1}$, therefore 
\begin{align}
    \Psi_{\underline{\chi}_{n}}^{GL_{n}}(a_1,\cdots,a_n)=&C_n\sum_{\substack{y'_{n-1,k}\in \mathbb{F}_{q}^{\times} \\ 1\leq k\leq n-1}}\chi_n\left((-1)^{n-1}\frac{\prod_{i=1}^{n}a_n}{\prod_{k=1}^{n-1}y'_{n-1,k}}\right)\psi\left(-\sum_{i=1}^{n-1}\frac{y'_{n-1,i}}{a_{i+1}}+\sum_{i=1}^{n-1}\frac{a_{n-i}}{y'_{n-1,n-i}}\right)\nonumber \\
    &  \times F_{n-1}\sum_{\substack{y_{k,i}\in \mathbb{F}_{q}^{\times} \\ 1\leq k\leq n-2\\ 1\leq i\leq k}}\chi_1(y'_{n-1,n-1}\prod_{k=i}^{n-2}y_{k,k})\prod_{i=2}^{n-1}\chi_i\left((-1)^{i-1}\frac{\prod_{k=i}^{n-2}y_{k,k-i+1}y'_{n,n-i}}{\prod_{j=1}^{i-1}y_{i-1,j}}\right)\nonumber \\
    &\times \psi\left(\sum_{i=1}^{n-1}\frac{y'_{n-1,n-i-1}}{y'_{n-1,n-i}}\sum_{k=i}^{n-2}\prod_{m=k}^{n-3}\frac{y_{m+1,m-i+1}}{y_{m,m-i+1}}\frac{1}{y_{n-2,n-i+1}}\right)\psi\left(-\sum_{k=1}^{n-2}\sum_{i=1}^{k}y_{k,i}\right)
\end{align}
In (4.9), let $y'_{n-1,k}=t_k$ for every $1\leq k\leq n-1$,  $y''_{k,i}=y_{k,i}$ for every $1\leq k\leq n-2$ and $1\leq i\leq k$, and $y''_{n-1,i}=1$ for every $1\leq i\leq n-1$, hence we obtain
\begin{align}
    \Psi_{\underline{\chi}_{n}}^{GL_{n}}(a_1,\cdots,a_n)=&C_n\sum_{\substack{t_{k}\in \mathbb{F}_{q}^{\times} \\ 1\leq k\leq n-1}}\chi_n\left((-1)^{n-1}\frac{\prod_{k=1}^{n}a_k}{\prod_{k=1}^{n-1}t_{k}}\right)\psi\left(-\sum_{k=1}^{n-1}\frac{t_{k}}{a_{k+1}}+\sum_{k=1}^{n-1}\frac{a_{k}}{t_{k}}\right) \nonumber \\
    &  \times F_{n-1}\sum_{\substack{y''_{k,i}\in \mathbb{F}_{q}^{\times} \\ 1\leq k\leq n-2\\ 1\leq i\leq k}}\chi_1(t_{n-1}\prod_{k=i}^{n-2}y''_{k,k})\prod_{i=2}^{n-1}\chi_i\left((-1)^{i-1}\frac{\prod_{k=i}^{n-2}y''_{k,k-i+1}t_{n-i}}{\prod_{j=1}^{i-1}y''_{i-1,j}}\right) \nonumber \\
    &\times \psi\left(\sum_{i=1}^{n-2}\frac{t_{n-i}}{t_{n-i+1}}\sum_{k=i}^{n-2}\prod_{m=k}^{n-2}\frac{y''_{m+1,m-i+1}}{y''_{m,m-i+1}}\right)\psi\left(-\sum_{k=1}^{n-2}\sum_{i=1}^{k}y''_{k,i}\right)
\end{align}
Note that the inner summation over $y''_{k,i}$ is exactly $\Psi_{\underline{\chi}_{n-1}}^{GL_{n-1}}(t_1,\cdots,t_{n-1})$.
Hence the recursion formula (4.1) is proved.
\end{proof}

\

\section{Proof of Proposition 4.1}

In Proposition 4.1, since $a$ is a diagonal matrix, therefore the $i$-th column of $ua$ equals to the $i$-th column of $u$ multiplied by $a_i$. Since right multiplication by $w_0$ simply permutes the column of $ua$, thus we deduce that it suffices to prove the following proposition.
\begin{prop}
Following the notations adopted in Proposition 4.1, for any given $n\geq 2$, let $u^{(n)}=\prod_{k=1}^{n-1}(I_n +\sum_{i=1}^{k}y_{k,i}E_{i,i+1})$, and $w_0^{(n)}=\left(\begin{smallmatrix}
   0 & &1 \\
     & \iddots &\\
1    & & 0
\end{smallmatrix}\right)$ we have the following formulas of minors:
\begin{align}
    \Delta_i(u^{(n)}w_0^{(n)})=(-1)^{\frac{i(i-1)}{2}}\prod_{k=i}^{n-1}\prod_{j=k-i+1}^{k}y_{k,j},
\end{align}
and
\begin{align}
    \frac{\Delta_{i,i+1}(u^{(n)}w_0^{(n)})}{\Delta_i(u^{(n)}w_0^{(n)})}=\sum_{k=i}^{n-1}\prod_{m=k}^{n-1}\frac{y_{m+1,m-i+1}}{y_{m,m-i+1}},
\end{align}
where $y_{n,k}=1$ is assumed for any $1\leq k\leq n-1$.
\end{prop}

\ 

Let $w_0^{(n-1)}=\left(\begin{array}{c|ccc}
  \begin{smallmatrix}
   0 & &1 \\
     & \iddots &\\
1    & & 0
\end{smallmatrix}
  & \begin{smallmatrix}
  0 \\ \vdots \\ 0
  \end{smallmatrix} \\
\hline 
  \begin{smallmatrix}
  0&  \cdots& 0
  \end{smallmatrix} &
  \begin{smallmatrix}
  1
  \end{smallmatrix}
\end{array}\right)$
, where the first block is a $(n-1)\times(n-1)$ matrix with 1s on the anti-diagonals and 0 elsewhere, and $u^{(n-1)}=\prod_{k=1}^{n-2}(I_n +\sum_{i=1}^{k}y_{k,i}E_{i,i+1})$. Note that here both $w_0^{(n-1)}$ and $u^{(n-1)}$ are $n\times n$ matrices, however essentially we are working with their first $(n-1)\times (n-1)$ blocks. In the following two theorems we have found recursive relations for the minors of $u^{(n)}w_0^{(n)}$ in terms of $n$, and Proposition 5.1 follows by resolving this recursive relation.

\begin{te}
\begin{align}
    \Delta_i(u^{(n)}w_0^{(n)})=\prod_{k=1}^{i}y_{n-1,n-k}\Delta_i(u^{(n-1)}w_0^{(n-1)}).
\end{align}
\end{te}

\begin{te} For every $1\leq i\leq n-1$ we have
\begin{align}
    \frac{\Delta_{i,i+1}(u^{(n)}w_0^{(n)})}{\Delta_i(u^{(n)}w_0^{(n)})}=\frac{1}{y_{n-1,n-i}}+\frac{y_{n-1,n-i-1}}{y_{n-1,n-i}}\frac{\Delta_{i,i+1}(u^{(n-1)}w_0^{(n-1)})}{\Delta_i(u^{(n-1)}w_0^{(n-1)})},
\end{align}
where $y_{n-1,0}=0$ is assumed.
\end{te}

To prove the above recursive relations we need a special case of the Cauchy-Binet formula, which is stated as Theorem 5.3 below.

\begin{te} (Cauchy-Binet formula)
Let $M$ and $N$be two $n\times n$ matrices (over arbitrary fields), and $I$, $J$ be two subsets of $\{1,\cdots,n\}$ with $i$ elements. Write $M_{IJ}$ for the $i\times i$ matrix whose rows are the rows of $M$ at indices from $I$ and columns are the columns of $M$ at indices from $J$, and $|M_{IJ}|$ for the determinant of this matrix. Therefore we have
\begin{align}
    |(MN)_{IJ}|=\sum_{K}|M_{IK}||N_{KJ}|,
\end{align}
where $K$ runs through all subsets of $\{1,\cdots,n\}$ with $i$ elements.
\end{te}

A proof of Theorem 5.3 can be found in Chapter 10 of \cite{sh}. Using formula (5.5) we are able to prove both Theorem 5.1 and Theorem 5.2.

\begin{proof}(Theorem 5.1)
Let $\fX_{n-1}^{n}=I_n +\sum_{i=1}^{n-1}y_{n-1,i}E_{i,i+1}$, therefore we have 
$$
u^{(n)}w_0^{(n)}=u^{(n-1)}w_0^{(n-1)}(w_0^{(n-1)})^{-1}\fX_{n-1}^{n}w_0^{(n)}.
$$
Let $A=u^{(n-1)}w_0^{(n-1)}$ and $B=(w_0^{(n-1)})^{-1}\fX_{n-1}^{n}w_0^{(n)}$, we see $A$ has the following shape:
$$
A=\left(\begin{matrix}
     *&\cdots &* &1 &0 \\
    \vdots &\iddots &1 &0 &0\\
   *& \iddots &\iddots  &\vdots &\vdots\\
1   &0 &\cdots & 0 &0 \\
0 &0 & \cdots & 0& 1
\end{matrix}\right).
$$
Furthermore,
$$
B=\begin{pmatrix}
y_{n-1,n-1} & 1 &0 &\cdots &0 \\
  0&  y_{n-1,n-2} &1 &\ddots &\vdots \\
 \vdots&\ddots  & \ddots & \ddots &0\\
  0& \cdots  &0 &y_{n-1,1} & 1\\
  1&0 &\cdots &0 &0 
\end{pmatrix},
$$
i.e, $y_{n-1,n-i}$ on the $i$-th place of the diagonal, 1 on the super-diagonal and the first entry of the $n$-th row, and 0 elsewhere. Now we use Theorem 5.3 to calculate $\Delta_i(AB)$,.

\

Let $I=\{1,\cdots,i\}$, we have $\Delta_i(AB)=|(AB)_{II}|$, therefore by Theorem 5.3 we have 
\begin{align}
    \Delta_i(AB)=\sum_{K}|A_{IK}||B_{KI}|,
\end{align}
where $K$ runs through all subsets of $\{1,\cdots,n\}$ with $i$ elements.
If $K=I$, then we have both $|A_{II}|$ and $|B_{II}|$ are the principal $i$-th minor of $A$ and $B$ respectively, therefore
$$
|A_{II}||B_{II}|=\prod_{k=1}^{i}y_{n-1,n-k}\Delta_i(u^{(n-1)}w_0^{(n-1)}).
$$
Suppose there exists $j\in K$ with $i<j<n$. Since the first $i$ entries in the $j$-th row of $B$ are zero, therefore we have $|B_{KI}|=0$ in this case. If $n\in K$, then the first $i$ entries in the last column of $A$ are zero, therefore $|A_{KI}|=0$ in this case. Hence we conclude that the only non-zero term in (5.6) is $|A_{II}||B_{II}|$. The proof of Theorem 5.1 is now complete.
\end{proof}

\

By a similar argument we could also prove Theorem 5.2.

\begin{proof} (Theorem 5.2)
Remark that by Theorem 5.1 the equation (5.4) is equivalent to the following
\begin{align}
    \Delta_{i,i+1}(u^{(n)}w_0^{(n)})&=\prod_{k=1}^{i-1}y_{n-1,n-k}\Delta_i(u^{(n-1)}w_0^{(n-1)}) \nonumber \\
    &+y_{n-1,n-i-1}\prod_{k=1}^{i-1}y_{n-1,n-k}\Delta_{i,i+1}(u^{(n-1)}w_0^{(n-1)}).
\end{align}
Let $I'=\{1,\cdots,i-1\}\cup \{i+1\}$, following the same notations adopted in Theorem 5.1 and Theorem 5.3, we have 
\begin{align}
    \Delta_{i,i+1}(AB)=\sum_{K}|A_{IK}||B_{KI'}|,
\end{align}
where $K$ runs through all subsets of $\{1,\cdots,n\}$ with $i$ elements.

If $K=I$, then clearly we have $|A_{IK}|=|A_{II}|=\Delta_i(u^{(n-1)}w_0^{(n-1)})$ and $|B_{KI'}|=|B_{II'}|=\prod_{k=1}^{i-1}y_{n-1,n-k}$ (the $(i,i)$ entry in $B_{II'}$ is 1).

Suppose we have $K=I'$, therefore $|A_{IK}|=|A_{II'}|=\Delta_{i,i+1}(u^{(n-1)}w_0^{(n-1)})$ and $|B_{KI'}|=|B_{I'I'}|=y_{n-1,n-i-1}\prod_{k=1}^{i-1}y_{n-1,n-k}$.

Suppose there exists $j\in K$ with $i+1<j\leq n$, then by exactly the same argument described in the proof of Theorem 6.1, we have $|A_{IK}||B_{KI'}|=0$ in this case.

One case for the choice of $K$ remaining is when both $i$ and $i+1$ are in $K$, in this case we see that the last two rows of $B_{KI'}$ are $(0,\cdots, 0, 1)$ and $(0,\cdots, 0, y_{n-1,n-i-1})$, therefore we obtain $|B_{KI'}|=0$.

Hence we conclude that only $|A_{II}||B_{II'}|$ and $|A_{II'}||B_{I'I'}|$ have contributions to the right hand side of (5.8), whence the proof of Theorem 5.3 is complete.

\end{proof}

\end{document}